\documentclass[]{amsart}

\def\bn{\mathbb{N}}
\def\bc{\mathbb{C}}
\def\h{\mathcal{H}}
\def\k{\mathcal{K}}

\newcommand{\ort}[1]{#1^{\perp}}
\newcommand{\inner}[2]{\langle #1,#2\rangle}
\newcommand{\sumjn}{\sum_{j=1}^n}
\newcommand{\sumin}{\sum_{i=1}^n}
\newcommand{\bh}{{\rm B}(\mathcal{H})}
\newcommand{\cb}[1]{{\rm CB}(#1)}
\newcommand{\matn}[1]{{\rm M}_{n}(#1)}
\newcommand{\normc}[1]{\overline{\overline{#1}}}
\newcommand{\weakc}[1]{\overline{#1}}
\newcommand{\e}[1]{{\rm E}(#1)}
\newcommand{\ue}[1]{\overline{\overline{{\rm E}(#1)}}}
\newcommand{\icb}[1]{\rm ICB{(#1)}}
\newcommand{\ib}[1]{{\rm IB}(#1)}
\newcommand{\hg}{\stackrel{h}{\otimes}}

\newcommand{\id}[1]{\mathcal{I}_{\rm c}({#1})}

\newtheorem{theorem}{Theorem}[section]
\newtheorem{lemma}[theorem]{Lemma}

\theoremstyle{remark}
\newtheorem{remark}[theorem]{Remark}
\theoremstyle{definition}

\numberwithin{equation}{section}

\begin{document}
\title[]{Uniform approximation by elementary operators} 
\author{Bojan Magajna} 
\address{Department of Mathematics\\ University of Ljubljana\\
Jadranska 21\\ Ljubljana 1000\\ Slovenia}
\email{Bojan.Magajna@fmf.uni-lj.si}
\thanks{{\em Acknowledgment.} The author is grateful to Ilja Gogi\' c for his comments on
an earlier version of the paper and to the anonymous referee for his suggestions.}


\begin{abstract}On a separable C$^*$-algebra $A$ every (completely) bounded map which preserves 
closed two sided ideals can be approximated uniformly by elementary operators
if and only if $A$ is a finite direct sum of C$^*$-algebras of continuous sections vanishing
at $\infty$ of locally trivial C$^*$-bundles of finite type.
\end{abstract}

\maketitle

\section{Introduction and the main result}

Throughout this paper $A$  will denote a C$^*$-algebra, $A_+$ the  positive and
$A_h$ the self-adjoint part of $A$. An {\em elementary
operator on $A$} is a map of the form
\begin{equation}\label{01}\psi(x)=\sum_{i=1}^m a_ixb_i\ \ (x\in A),\end{equation}
where $a_i$ and $b_i$ are fixed elements of the multiplier algebra $M(A)$ of $A$. The
smallest $m$ for which $\psi$ can be expressed in the form (\ref{01}) is called the
{\em length} of $\psi$.
The set of all elementary operators on $A$
is denoted by $\e{A}$ and its norm closure (in the set of all bounded operators on $A$)
by $\ue{A}$. By $\id{A}$ we will 
denote the set of all closed two-sided ideals in $A$ and by $\ib{A}$ (resp. $\icb{A}$)
the set of all bounded (resp. all completely bounded \cite{P})
maps that preserve all ideals in $\id{A}$. By an {\em ideal} we shall always mean a closed
two-sided ideal. Clearly $\e{A}\subseteq\icb{A}\subseteq\ib{A}$.
In this note we characterize C$^*$-algebras for which the equalities $\icb{A}=\ue{A}$
or $\ib{A}=\ue{A}$ hold. 

\begin{theorem}\label{th11} For a separable C$^*$-algebra $A$  the inclusion 
$\icb{A}\subseteq\ue{A}$ 
holds if and only if $A$ is a finite direct sum of
homogeneous C$^*$-algebras of finite type;  in this case $\ib{A}=\e{A}=\icb{A}$.
\end{theorem}

We recall that a C$^*$-algebra $A$ is called {\em $n$-homogeneous} if all its irreducible
representations are of the same finite dimension $n$. 
(By the dimension of a representation $\pi$ we mean the dimension of the Hilbert space of $\pi$.)
Then by 
\cite{Fe}, \cite{TT} $A$ is isomorphic to the C$^*$-algebra $\Gamma_0(E)$ of all continuous
sections vanishing at $\infty$ of a locally trivial C$^*$-bundle $E$ with fibers isomorphic
to $\matn{\bc}$. ($E$ is just a usual vector bundle
such that the local trivializations, restricted to
fibers, are 
isomorphisms of C$^*$-algebras.) If the base space $\Delta$ of this bundle admits a
finite open covering $(\Delta_i)$ such that $E|\Delta_i$ is trivial for each $i$ (as a
C$^*$-bundle), then
$E$ is said to be of {\em finite type} \cite{Hu} and we shall say that in this case $A$
is of {\em finite type}.

We note that a weaker form of approximation is always possible: namely, for every 
C$^*$-algebra $A$ the set  $\e{A}$ is dense in $\icb{A}$ (and 
in $\ib{A}$) in the 
point norm - topology
\cite[2.3]{M1}, \cite[5.3.4]{AM}. However, there is in general no control on the norms in this approximation:
not every complete contraction $\phi\in\icb{A}$ can be approximated by a net of complete
contractions in $\e{A}$. If $A$ is a von Neumann algebra, then each  $\phi\in\cb{A}$ preserving
all weak* closed ideals can be approximated by complete contractions in $\e{A}$ in the point-
weak* topology if and only if $A$ is injective \cite{CS} (at least if the predual of $A$
is separable). For a general C$^*$-algebra $A$,
the question of when every complete contraction $\phi\in\icb{A}$ can be approximated 
pointwise by complete contractions in $\e{A}$ is connected to the theory of tensor products
of C$^*$-algebras and the complete answer is still not clear to the author.
Concerning elementary operators, we mention that in recent years the interest has shifted 
from spectral and structural theory (\cite{Cu}, \cite{Fi})
to questions related to the natural map $\mu$ from
the central Haagerup tensor product $M(A)\hg_Z M(A)$ into $\cb{A}$
(see \cite{S}, \cite{CS}, \cite{AM1}, \cite{M2}, \cite{AM} and the references there). In 
particular, the problem of when $\mu$ is isometric has been much studied by several authors
for special cases of C*-algebras (see \cite[Chapter 5]{AM})
and has been finally solved for general C*-algebras by Archbold, Somerset and Timoney in 
\cite{AST} and \cite{So}. Clearly the range of  $\mu$ is contained in $\ue{A}$
and the above Theorem \ref{th11} characterizes C$^*$-algebras in which the range of $\mu$
is as large as possible.

In one direction the proof of Theorem \ref{th11} is easy. Namely, if $A=\Gamma_0(E)$ with 
$E$ of finite type, the usual (finite) partition of unity argument reduces 
the proof to the case when $E\cong\Delta\times\matn{\bc}$ is trivial, so that $A\cong C_0(
\Delta,\matn{\bc})$ (continuous matrix valued functions vanishing at $\infty$). In this 
special
case a bounded linear map $\phi$, which preserves all ideals of the form $J_t=\{f\in A:\ f(t)=0\}$
($t\in\Delta$), decomposes into a bounded continuous collection of maps on fibers
$A/J_t\cong\matn{\bc}$; in other words, $\phi$ is in the set ${\rm B}_{Z_0}(A)=
{\rm B}_{Z_0}(\matn{Z_0})$ of bimodule maps over the center $Z_0=C_0(\Delta)$
of $A$. With $e_{ij}$ the standard matrix units in $\matn{\bc}$ and $\eta_{kl}:\matn{Z_0}\to Z_0$
the maps $\eta_{kl}([z_{ij}])=z_{kl}$, we have that
$$\phi^{jl}_{ki}:Z_0\to Z_0,\ \ \phi^{jl}_{ki}(z):=\eta_{kl}(\phi(z e_{ij}))\ \ (z\in Z_0)$$
are bimodule maps over $Z_0$ (thus, double centralizers since $Z_0$ is commutative), 
hence given by multiplications with certain elements
$c^{jl}_{ki}$ of the multiplier algebra $Z=C_b(\Delta)$ of $Z_0$. Then 
for $[z_{ij}]=\sum_{i,j=1}^{n}z_{ij} e_{ij}$ we have that
$$\phi([z_{ij}])=\sum_{i,j=1}^n\phi(z_{ij} e_{ij})=\sum_{i,j,k,l=1}^nc^{jl}_{ki}z_{ij}
e_{kl}=\sum_{k,l,r,s=1}^nc^{rl}_{ks}e_{ks}(\sum_{i,j=1}^nz_{ij}e_{ij})e_{rl},$$
so that $\phi$ is an elementary operator with coefficients in $\matn{Z}$, the multiplier 
C$^*$-algebra of $A=\matn{Z_0}$.

In certain special cases (say, if $A$ is prime) one can use the Akemann-Pedersen
characterization of C$^*$-algebras having only inner derivations \cite{AP} together with some
additional work to give a relatively short proof of a part 
of Theorem \ref{th11}. But in general  
the proof of Theorem \ref{th11}  requires construction 
of new classes of maps preserving ideals, which can not be uniformly approximated by elementary
operators. It is perhaps not very surprising that such maps exist if the dimensions of irreducible
representations of $A$ are not bounded. They can be taken to be of the
form $x\mapsto\sum_{k=1}^{\infty}e_kxf_k$, where the sum is norm convergent for all
$x\in A$, but not uniformly convergent. It will be shown in Section 2 that an appropriate
choice of the coefficients $e_k$ and $f_k$ is possible so that such a map can not
be approximated uniformly by elementary operators. 

But even if $A$ is subhomogeneous 
we do not always have the inclusion $\icb{A}\subseteq\ue{A}$. Consider, for example, the 
C$^*$-subalgebra $A_0$ of $C([0,1],{\rm M}_2(\bc))$ consisting of all $x$ such that
$x(0)$ is a diagonal matrix with $0$ on the $(2,2)$ position (or, alternatively, a general 
diagonal matrix) and the map $x\mapsto\phi(x)=e_{12}xe_{12}$, where $e_{12}$ has $1$ on the position
$(1,2)$ and $0$ elsewhere. It can be shown that $\phi$ preserves ideals but $\phi\notin\ue{A_0}$
(the details are just a special case of those in the proof of Lemma \ref{le28}, the paragraphs containing 
(\ref{33})--(\ref{37})). Such examples  suggest the way to a part of the proof of 
Theorem \ref{th11}.  
Namely, for a $n$-subhomogeneous C$^*$-algebra $A$ 
which is not a direct sum
of homogeneous C$^*$-algebras, it will be shown in Section 4 that the multiplier algebra
$M(J)$  of the $n$-homogeneous ideal $J$ of $A$ contains an element $b$ such that the
twosided multiplication $\phi:x\mapsto bxb$ maps $A$ into $A$ and $\phi\in\icb{A}\setminus
\ue{A}$ (provided that $J$
is essential in $A$, then the general case will be reduced to this situation). As a preparation
for this, we shall show in Section 3 that, if $J$ is not unital, $M(J)$ is the C$^*$-algebra 
of continuous sections of a (not necessarily locally trivial)
C$^*$-bundle over the Stone - Chech compactification $\beta(U)$ of the  spectrum $U$
of $J$ and the $n$-homogeneous ideal of $M(J)$ properly contains $J$ (Lemma \ref{le27}).
This will enable us to show in Section 4 (as the first step towards the proof of Theorem 
\ref{th11} in the case of subhomogeneous C$^*$-algebras) that there exists a point in 
$\beta(U)$ at which $A$ ($\subseteq M(J)$) looks
in a certain respect essentially like $A_0$ of the example mentioned above.

On the other hand, the
explanation that the homogeneous summands in Theorem \ref{th11} must be of finite type 
is simple and can be given right now.

\begin{proof}[Proof that $A$ must be of finite type] Assume that  a locally trivial C$^*$-bundle $E$ over a locally compact space $\Delta$
with fibers $\matn{\bc}$ is not of finite type. Then $E$
is not of finite type as a vector bundle by \cite[2.9]{Ph} and it follows that for any finite
set $\{a_1,\ldots,a_m\}$ of bounded continuous sections of $E$ there exists a point $t_0\in\Delta$
such that $$\dim{{\rm span}\{a_1(t_0),\ldots,a_m(t_0)\}}<n^2.$$
Indeed, if this was not the case, then the map $f:\Delta\times\bc^m\to E$,
$f(t;\lambda_1,\ldots,\lambda_m)=\sum_{j=1}^m\lambda_ja_j(t)$, would be a surjective
morphism of vector bundles and $E$ would be isomorphic to the subbundle $(\ker f)^{\perp}$
of $\Delta\times\bc^m$, hence of finite type by \cite[3.5.8]{Hu}. It follows that
for each elementary operator $\psi$ on $A=\Gamma_0(E)$ there is a point $t_0\in\Delta$
such that the induced elementary operator $\psi_{t_0}$ on $A/J_{t_0}\cong\matn{\bc}$ has
the length at most $n^2-1$. On the other hand the (normalized) central trace $\tau$ on $A$
(defined by $\tau(x)(t)=(1/n){\rm tr}\,x(t)$, $t\in\Delta$) preserves all (primitive ideals
$J_t$ hence all) ideals of $A$, hence $\tau\in\icb{A}$. But, denoting by $e_{i,j}$
($i,j=1,\ldots,n$) the usual matrix units in $\matn{\bc}$, we have that
$\tau(x)(t_0)=(1/n)\sum_{i,j=1}^ne_{i,j}x(t)e_{j,i}$, so that $\tau_{t_0}$ on $\matn{\bc}$ has
length $n^2$. Therefore $\tau_{t_0}$ has a positive distance $d$ to the closed set of all
elementary operators of length $\leq n^2-1$ on $\matn{\bc}$. This implies that the distance
of $\tau$ to $\e{A}$ is at least $d$, so $\tau\notin\ue{A}$.
\end{proof}
 
Throughout this paper we shall denote by $\hat{A}$ the spectrum of $A$ (= the set of all
equivalence classes of irreducible representations) and by $\check{A}$
the primitive spectrum of $A$ (= the set of all primitive ideals), equipped with the 
Jacobson topology.   
The norm and the weak* closure of a set $S$ will be denoted by $\normc{S}$ and
$\weakc{S}$, respectively.

\section{A reduction to subhomogeneous C$^*$-algebras}

\begin{lemma}\label{le21} Let $A$ be an irreducible C$^*$-subalgebra in $\bh$,
$x_1,\ldots,x_n$ arbitrary elements of $A$, $g\in A_+\setminus\{0\}$, $B=\normc{gAg}$
the hereditary C$^*$-subalgebra generated by $g$ and $\varepsilon>0$. If ${\rm rank}\,g>n$
then there exist
$e,f\in B_+$ such that $\|e\|=1=\|f\|$ and
$$\|ex_jf\|<\varepsilon\ \ (j=1,\ldots,n).$$
\end{lemma}

\begin{proof} Choose a unit vector 
$\eta\in\k:=[B\h]$. Note that $b=(b|\k)\oplus(0|\k^{\perp})$ for each $b\in B$ 
(since $B\ort{\k}
\subseteq\k\cap\ort{\k}=0$), that $B|\k$ is irreducible \cite[5.5.2]{Mu} and
$\dim\k>n$ since ${\rm rank}\,g>n$. Hence by the Kadison transitivity theorem there exists
$e\in B$ with $\|e\|=1$ such that $e$ annihilates the projections of all vectors
$x_j\eta$ to $\k$. Thus (since $e\k^{\perp}=0$)
$$ex_j\eta=0\ \ (j=1,\ldots,n).$$
Moreover, replacing $e$ by $e^*e$, we may assume that $e\in B_+$. By the algebraic 
irreducibility \cite[5.2.3]{Mu} there
exists $b\in B_+$ with $\|b\|=1$ and $b\eta=\eta$. Then the vector state $\omega_{\eta}(x)
:=\inner{x\eta}{\eta}$ annihilates the element
$$x_0:=\sumjn bx^*_je^2x_jb$$
of $B$. Since $\omega_{\eta}$ is a pure state on $B$ (by irreducibility of $B$ on $\k$),
by \cite{AAP} there exists a positive element $h$ in the unitisation of $B$ such that $\|h\|=1$, 
$$\|hx_0h\|<\varepsilon^2\ \ \mbox{and}\ \ \omega_{\eta}(h)=1.$$
This implies that $\|ex_jbh\|<\varepsilon$ for all $j=1,\ldots,n$, and (since  
$\|h\|=1$ and $\inner{h\eta}{\eta}=1$) $h\eta=\eta$. Set $f:=|hb|=|(bh)^*|$. Then
$f\in B_+$, $\|f\|=1$ (since $hb\eta=\eta$) and (using the polar decomposition 
$(bh)^*=uf$ of $(bh)^*$) we deduce that $\|ex_jf\|=\|ex_jbh\|<\varepsilon$ for all $j=1,\ldots,n$.
\end{proof}

\begin{lemma}\label{le22} If $A$ is separable and has an infinite dimensional irreducible
representation $\pi:A\to\bh$, then there exist two bounded sequences $(e_i)$ and $(f_i)$
in $A_+$ such that $\|\pi(e_i)\|=1=\|\pi(f_i)\|$, $e_ie_j=0=f_if_j$ if $i\ne j$ and the sum
\begin{equation}\label{21}\phi(x)=\sum_{n=1}^{\infty}e_nxf_n\end{equation}
is norm convergent for each $x\in A$.
\end{lemma}

\begin{proof} Since $\pi(A)$ is irreducible and $\h$ infinite dimensional, $\pi(A)$ must be
infinite dimensional and the same for any of its maximal abelian self-adjoint subalgebras \cite[4.6.12]{KR}.
Thus, by functional calculus we may find a sequence $(\dot{g}_i)$ in $\pi(A)_+$ with
$\|\dot{g}_i\|=1$, $\dot{g}_i\dot{g}_j=0$ if $i\ne j$ and ${\rm rank}\, \dot{g}_i>i$.
Set $P=\ker\pi$ and identify $A/P$ with $\pi(A)$. Let $(x_j)$ be a bounded sequence 
with dense span in $A$. By Lemma \ref{le21} for each $n$ there exist elements $\dot{e}_n$
and $\dot{f}_n$ in $\normc{(\dot{g}_n\pi(A)\dot{g}_n)}_+$ such that $\|\dot{e}_n\|=1=
\|\dot{f}_n\|$ and
\begin{equation}\label{22}\|\dot{e}_n\pi(x_j)\dot{f}_n\|<\frac{1}{2^n}\ \ (j=1,\ldots,n).
\end{equation}
Since the $\dot{g}_n$'s are orthogonal (that is, $\dot{g}_i\dot{g}_j=0$ if $i\ne j$), the
same holds for $\dot{e}_n$ and for $\dot{f}_n$. By \cite[4.6.20]{KR} we may lift $(\dot{e}_n)$
(and similarly $(\dot{f}_n)$) from $\pi(A)$ to orthogonal sequences $(\tilde{e}_n)$ 
(and $(\tilde{f}_n)$) of  norm $1$ elements
in $A_+$. Recall that, with $(u_k)$ an approximate unit in $P$, we have $\|\pi(x)\|=\lim
\|(1-u_k)x(1-u_k)\|$ for all $x\in A$, hence, from (\ref{22}) for each $n$ there exists
$u_n\in P$, $0\leq u_n\leq1$, such that
\begin{equation}\label{23}\|(1-u_n)\tilde{e}_nx_j\tilde{f}_n(1-u_n)\|<\frac{1}{2^n}\ \
(j=1,\ldots,n).\end{equation}
Set $e_n=\tilde{e}_n(1-u_n)\tilde{e}_n$ and $f_n=\tilde{f}_n(1-u_n)\tilde{f}_n$. Then
$e_ie_j=0=f_if_j$ if $i\ne j$,  $\|e_n\|=1=\|f_n\|$ (since $\|\pi(e_n)\|=1$,  
$\|\pi(f_n)\|=1$ and $\|e_n\|, \|f_n\|\leq1$), and (\ref{23}) implies that
\begin{equation}\label{24}\|e_nx_jf_n\|<\frac{1}{2^n}\ \ (j=1,\ldots,n).\end{equation}

Since $\|\sum_{n=1}^{\infty}e_n^2\|=\max_n\|e_n^2\|=1$ (by orthogonality) and 
$\|\sum_{n=1}^{\infty}f_n^2\|=1$, it follows that (\ref{21}) defines a (complete)
contraction $\phi$ from $A$ into the von Neumann envelope $\weakc{A}$ of $A$. We have 
$$\phi(x_j)=\sum_{n=1}^{j-1}e_nx_jf_n+\sum_{n=j}^{\infty}e_nx_jf_n,$$
where the sum on the right side is norm convergent by (\ref{24}). Since the sequence
$(x_j)$ has dense span in $A$ it follows that the sum (\ref{21}) is convergent
for each $x\in A$.
\end{proof}

If $p_i\in\bh$ ($i=1,\ldots,n$) are nonzero orthogonal projections and $\phi\in\e{\bh}$ is defined by 
$\phi(x)=\sum_{i=1}^np_ixp_i$, the distance of $\phi$ to the set $E_{n-1}$ of elementary operators
of length at most $n-1$ turns out to be (not $1$, but) at most $1/n$. (For a proof, let 
$\psi\in\e{\bh}$ be defined by $\psi(x)=\sum_{i,j=1}^n(\delta_{i,j}-\frac{1}{n})p_ixp_j$ and note that
$\phi(x)-\psi(x)=\frac{1}{n}pxp$, where $p=\sum_{i=1}^np_i$, so that $\|\phi-\psi\|=\frac{1}{n}$. To show that the length 
of $\psi$ is at most $n-1$, observe that the $n\times n$ matrix $[\delta_{i,j}-\frac{1}{n}]$ is
(a projection) of  rank 
$n-1$, therefore there exist $\alpha_{i,j},\beta_{i,j}
\in\bc$ such that $\delta_{i,j}-\frac{1}{n}=\sum_{k=1}^{n-1}\alpha_{k,i}\beta_{k,j}$ for all
$i,j$. Now, with $a_k:=\sum_{i=1}^n\alpha_{k,i}p_i$ and $b_k:=\sum_{i=1}^n\beta_{k,i}p_i$ we have that
$\psi(x)=\sum_{i,j=1}^n\sum_{k=1}^{n-1}\alpha_{k,i}\beta_{k,j}p_ixp_j=\sum_{k=1}^{n-1}a_kxb_k$
for all $x\in\bh$.) However, we shall only need
an asymptotic estimate stated in the following lemma.

\begin{lemma}\label{le220} For each $m\in\bn$ there exists $n(m)\in\bn$ such that for every 
$\theta\in\e{\bh}$ of the form 
$$\theta(x)=\sumin e_ixf_i\ \ (x\in\bh),$$
where $n\geq n(m)$ and $e_i,f_i\in\bh_+$ are norm $1$ elements satisfying $e_ie_j=0=f_if_j$
if $i\ne j$, the distance $d(\theta,E_m)$ of $\theta$ to the set $E_m$ of all elementary
operators of  length at most $m$ is at least $1/5$.
\end{lemma}

\begin{proof}Denote by $\bh^{\sharp}$ the dual of $\bh$ and note that the map
$$\kappa:\e{\bh}\to{\rm B}(\bh^{\sharp},\bh),\ \ \kappa(\sumin a_i\otimes b_i)(\rho)=
\sumin\rho(a_i)b_i\ (\rho\in\bh^{\sharp})$$
is contractive, where the elements $\psi=\sum a_i\otimes b_i\in\e{\bh}$ have the usual
operator norm  $\|\psi\|=\sup\{\|\sum a_ixb_i\|:\ x\in\bh,\ \|x\|\leq1\}$. This
follows from
$$\|\kappa(\psi)\|=\sup\{|\sum\rho(a_i)\omega(b_i)|:\ \omega,\rho\in\bh^{\sharp},\ \|\omega\|
\leq1,\ \|\rho\|\leq1\},$$
by  noting first that the supremum does not change if we restrict $\omega$ and $\rho$ to be of
rank $1$ (since the unit ball of  $\bh^{\sharp}$ is the weak* closure of the
convex hull of rank one functionals of the form $x\mapsto\langle x\xi,\eta\rangle$, where 
$\xi,\eta\in\h$ have norm at most $1$) and then noting that the  supremum is  equal to
$$\sup\{\|\sum a_ixb_i\|:\ x\in\bh,\ \|x\|\leq1,\ \mbox{rank}\ x\leq1\},$$
hence dominated by $\|\psi\|$. 

Let $\theta$ be as in the Lemma (but with $n$ arbitrary). Given $\psi\in E_m$ of the form
$$\psi(x)=\sum_{j=1}^ma_jxb_j,$$
let $U$ be the closed unit ball of $V:={\rm span}\{b_1,\ldots,b_m\}$. By the orthogonality of
the $e_i$'s we may choose $\rho_i$ in the unit ball of $\bh^{\sharp}$ so that $\rho_i(e_j)=
\delta_{i,j}$, hence
$$\varepsilon:=\|\theta-\psi\|\geq\|\kappa(\theta)-\kappa(\psi)\|\geq
\|f_i-\sum_{j=1}^m\rho_i(a_j)b_j\|.$$
This shows that the distance of $f_i$ to $V$ is at most $\varepsilon$ and, since $\|f_i\|=1$,
it follows that ${\rm dist}(f_i,U)\leq2\varepsilon$. Thus, we may choose $h_i\in U$ with
$\|f_i-h_i\|\leq2\varepsilon$, from which we have (since $\|f_i-f_j\|=1$ if $i\ne j$)
$$\|h_i-h_j\|\geq\|f_i-f_j\|-\|f_i-h_i\|-\|f_j-h_j\|\geq1-4\varepsilon.$$
Suppose  that $\varepsilon<1/5$, so that $\|h_i-h_j\|>1/5$ for
all $i\ne j$. If we equip $V$ with a suitable Euclidean norm $\|\cdot\|_2$ (by proclaiming
an Auerbach basis of $V$ to be orthonormal), then $\|\xi\|/\sqrt{m}\leq\|\xi\|_2\leq\|\xi\|
\sqrt{m}$ for all $\xi\in V$. Thus, $\|h_i-h_j\|_2>1/(5\sqrt{m})$ if $i\ne j$, while all the
vectors $h_i$ ($i=1,\ldots,n$) are contained in the same at most $m$-dimensional Euclidean
ball of radius $\sqrt{m}$. This is clearly impossible if $n$ is large enough.
\end{proof}

\begin{lemma}\label{le23} Suppose that $A$ is separable. If $\icb{A}\subseteq\ue{A}$ then
$A$ is subhomogeneous, that is, $\sup_{[\pi]\in\hat{A}}\dim\pi<\infty$.
\end{lemma}

\begin{proof}First we will show that all irreducible representations of $A$  must
be finite dimensional. Suppose the contrary, that $\pi:A\to\bh$ is an infinite dimensional
irreducible representation and consider the map $\phi$ defined in Lemma \ref{le22}.
Clearly $\phi\in\icb{A}$. Denote by $\dot{\phi}$ the map on $\dot{A}:=A/\ker\pi$ induced by
$\phi$. From the norm convergent series (\ref{21}) we have that $\dot{\phi}(x)=
\sum_{n=1}^{\infty}\pi(e_n)x\pi(f_n)$ ($x\in\pi(A)$) and by the same formula $\dot{\phi}$
can be extended uniquely to a weak* continuous (complete) contraction $\weakc{\phi}$ on
$\bh$ (the weak* closure of $\pi(A)$). If $\phi\in\ue{A}$, then $\dot{\phi}\in\ue{\dot{A}}$
and (since the norm of any weak* continuous operator on $\pi(A)$ agrees with the norm of its
weak* continuous extension to $\weakc{\pi(A)}$, a consequence of the Kaplansky density 
theorem) $\weakc{\phi}\in\ue{\bh}$. Thus, there exists $\psi\in\e{\bh}$, say $\psi(x)=
\sum_{j=1}^{m} a_jxb_j$, such that 
\begin{equation}\label{230}\|\weakc{\phi}-\psi\|<\frac{1}{5}.\end{equation}
Now, for each $N\in\bn$ denote by $P_N$ and $Q_N$ the projections onto $\sum_{n=1}^N\pi(e_n)\h$
and $\sum_{n=1}^N\pi(f_n)\h$, respectively. From orthogonality of each of the sequences
$(e_n)$ and $(f_n)$, the operator $P_N\weakc{\phi}Q_N$ has the form
$P_N\weakc{\phi}Q_N(x)=\sum_{n=1}^N\pi(e_n)x\pi(f_n)$. But (\ref{230}) implies that
$\|P_N\weakc{\phi}Q_N-P_N\psi Q_N\|<1/5$ for all $N$ and, since $P_N\psi Q_N$ is an
elementary operator of length at most $m$, this contradicts Lemma \ref{le220}. 

Thus for each irreducible representation $\pi$ the C$^*$-algebra $\pi(A)$ is isomorphic to 
${\rm M}_r(\bc)$ for some $r\in\bn$,  we may identify $\hat{A}$ with $\check{A}$ and each primitive
ideal $P$ of $A$ is maximal. A point $P\in\check{A}$ is called {\em Hausdorff} (or separated) if for
each $Q\in\check{A}$, $Q\ne P$, there exist disjoint open neighborhoods of $P$ and $Q$ in
$\check{A}$. (Note that in our situation singletons are automatically closed sets since
primitive ideals are maximal.) By \cite[3.9.4]{Dix} the set $S$ of Hausdorff points is
dense in $\check{A}$. If $S$ is finite, then $S=\check{A}$, $A$ is finite dimensional and
the proof is finished in this case. So we may assume that $S$ is infinite. Since for each
$g\in A_+$ the trace function $[\pi]\mapsto {\rm tr}\, \pi(g)$ is lower semicontinuous on
$\hat{A}$ \cite{Pe}, the same holds for the rank function (for ${\rm rank}\, \pi(g)=
\sup_n{\rm tr}\,\sqrt[n]{\pi(g)}$ if $\|g\|\leq1$). Thus, if we assume that $\sup_{[\pi]\in
\hat{A}}\dim\pi=\infty$, then there exists a sequence $(\sigma_k)$ in $S$ with $\dim\sigma_k$
tending to $\infty$ as $k\to\infty$. Suppose first that there exists a limit point
$\sigma$ of $(\sigma_k)$ in $\hat{A}$.
Since $\sigma_1$ is a Hausdorff point, there exist disjoint open neighborhoods $U_1$ of 
$\sigma_1$ and $V_1$ of $\sigma$. Put $[\pi_1]=\sigma_1$ and choose any $[\pi_2]\in V_1\cap
(\sigma_k)$ such that $\dim\pi_2>2\cdot 3$. Since $[\pi_2]$ is a Hausdorff point, there exist
disjoint open neighborhoods $U_2\subseteq V_1$ of $[\pi_2]$ and $V_2\subseteq V_1$ of $\sigma$.
Continuing in this way, we find a sequence $([\pi_k])\subseteq\hat{A}$  
such that $\dim\pi_k>k(k+1)$, and open neighborhoods $U_k$ of $[\pi_k]$ and
$V_k$ of $\sigma$ such that $U_k\cap V_k=\emptyset$ and $U_{k+1},V_{k+1}\subseteq V_k$.
In particular $U_n\cap\cup_{k\ne n}U_k=\emptyset$, hence $[\pi_k]\notin U_n$ if $k\ne n$, 
which implies that the kernel $P_n$ of
$\pi_n$ is not contained in the closure of the set $\{P_k:\ k\ne n\}$. If the sequence
$(\sigma_k)$ has no limit points, then we simply let $([\pi_k])$ be a subsequence with
$\dim\pi_k>k(k+1)$ and then again $P_n=\ker\pi_n$ is not in the closure of $\{P_k:\ k\ne n\}$.
Setting $R_n=
\cap_{k\ne n}P_k$, this means that $P_n$ does not contain $R_n$, hence $P_n+R_n=A$ since $P_n$
is maximal. Since $\pi_n(A)$ is of the form ${\rm M}_r(\bc)$ for some $r>n(n+1)$,
there exist  mutually orthogonal projections  $\pi_n(g_{ni})$ ($i=1,\ldots,n)$ in $\pi_n(A)$ such that 
${\rm rank}\,\pi(g_{ni})>n$ and $\sumin\pi_n(g_{ni})=1$. These may be lifted to mutually
orthogonal positive contractions $g_{ni}$ in $A$ \cite[4.6.20]{KR}. Moreover, since $R_{n+}+
P_{n+}=A_+$
and $P_n=\ker\pi_n$, we may achieve that $g_{ni}\in R_n$. Set $\tilde{g}_n=\sumin g_{ni}$
and define recursively $g_1=\tilde{g}_1$, $g_n=(1-g_1-\ldots-g_{n-1})\tilde{g}_n(1-g_1
-\ldots-g_{n-1})$. Then $\sum_{n=1}^mg_n\leq1$ for all $m$ (by an induction, using that
$h^2\leq h$ if $0\leq h\leq1$), hence 
$\sum_{n=1}^{\infty}g_n\leq1$ (in the von
Neumann envelope of $A$) and $\pi_n(g_n)=1$ since $\pi_n(\tilde{g}_n)=1$ and 
$g_m\in R_m\subseteq P_n=\ker\pi_n$
if $m\ne n$.

Let $(x_j)$ be a sequence  with a dense span in $A$ and $\|x_k\|\leq1$. By Lemma \ref{le21} there exist
positive norm $1$ elements $\dot{e}_{ni}$ and $\dot{f}_{ni}$ in $\pi_n(g_{ni}Ag_{ni})$
such that 
\begin{equation}\label{25}\|\dot{e}_{ni}\pi_n(x_j)\dot{f}_{ni}\|<\frac{1}{n2^n}\ \ (i,j=1,\ldots,n).
\end{equation}
Note that $\sumin\dot{e}_{ni}\leq1$ (and similarly for $\dot{f}_{ni}$) by mutual orthogonality of the
projections $\pi_n(g_{ni})$ for a fixed $n$. For each $n$ we may lift $\dot{e}_{n1}$ to a positive element $e_{n1}$ in $A$ such that
$e_{n1}\leq g_n$ since $\dot{e}_{n1}\leq\pi_n(g_n)=1$ (see \cite[4.6.21]{KR}).
Assuming inductively that for some $i<n$ we already have elements $e_{nj}$ 
($j=1,\ldots,i$) in $A_+$ such that $\pi_n(e_{nj})=\dot{e}_{nj}$ and 
$e_{n1}+\ldots+e_{ni}\leq g_n$, then
by \cite[4.6.21]{KR} we may find $e_{n,i+1}$ in $A_+$ such that $\pi_n(e_{n,i+1})=\dot{e}_{n,i+1}$ and 
$e_{n,i+1}\leq g_n-(e_{n1}+\ldots+e_{n,i})$ since $\dot{e}_{n,i+1}\leq1-\dot{e}_{ni}-
\ldots-\dot{e}_{ni}=\pi_n(g_n-e_{n1}-
\ldots-e_{ni})$. Thus we may find $e_{n,i}$ so that $\sumin e_{ni}\leq g_n$ and it follows that 
\begin{equation}\label{26}\sum_{n=1}^{\infty}\sumin e_{ni}\leq1.\end{equation}
Similarly, there exist elements $f_{ni}\in A_+$ such that $\pi_n(f_{ni})=\dot{f}_{ni}$
and $$\sum_{n=1}^{\infty}\sumin f_{ni}\leq1.$$

Given $u\in P_n$ with $0\leq u\leq1$, we may replace  the elements $e_{ni}$ ($i=1,
\ldots,n$, $n$ fixed) by $|(1-u)e_{ni}|^2$ without violating (\ref{26}) (since 
$e_{ni}(1-u)^2e_{ni}\leq e_{ni}^2\leq e_{ni}$). 
Choosing $u$ from an (increasing) approximate unit of $P_n$, we have from (\ref{25}) that
$$\begin{array}{lll}\inf_u\|e_{ni}(1-u)^2e_{ni}x_jf_{ni}\|&\leq&
\lim_u\|(1-u)e_{ni}x_jf_{ni}\|\\
&=&\|\dot{e}_{ni}\pi_n(x_j)\dot{f}_{ni}\|\\
&<&\frac{1}{n2^n}\ \ (i,j=1,\ldots,n).\end{array}$$
Thus, we may assume that $e_{ni}$ and $f_{ni}$ have been chosen so that (note that
$e^2\leq e$ if $0\leq e\leq1$)
\begin{equation}\label{27}\sum_{n=1}^{\infty}\sumin e_{ni}^2\leq1,\ \
\sum_{n=1}^{\infty}\sumin f_{ni}^2\leq1,\end{equation}
\begin{equation}\label{28}\pi_n(e_{ni})\pi_n(e_{nj})=0=\pi_n(f_{ni})\pi_n(f_{nj})\ \mbox{if}\ 
i\ne j,\ \ \|\pi_n(e_{ni})\|=1=\|\pi_n(f_{ni})\| \end{equation}
\begin{equation}\label{29}\mbox{and}\ \ \ \ \ \|e_{ni}x_jf_{ni}\|<\frac{1}{n2^n}\ (i,j=1,\ldots,n).
\end{equation}

By (\ref{27}) we may define a (complete) contraction $\phi:A\to\weakc{A}$ by
\begin{equation}\label{210}\phi(x)=\sum_{n=1}^{\infty}\sumin e_{ni}xf_{ni}\ \ (x\in A).
\end{equation}
Since the sequence $(x_j)$ has dense span in $A$, (\ref{29}) implies that the series
(\ref{210}) is norm convergent for each $x\in A$, consequently $\phi\in\icb{A}$.

If $\|\phi-\psi\|<1/5$ for some $\psi\in\e{A}$ of length (say) $m$, then also 
\begin{equation}\label{271}\|\phi_n-\psi_n\|<1/5,\end{equation}
where $\phi_n$ and $\psi_n$ are the maps on 
$A_n:=\pi_n(A)\cong A/P_n\cong{\rm M}_{r(n)}(\bc)$
induced by $\phi$ and $\psi$ (respectively). Since $\pi_n(e_{mi})=0$ if $m\ne n$ (for $e_{mi}\leq g_m$),
$$\phi_n(\dot{x})=\sumin\dot{e}_{ni}\dot{x}\dot{f}_{ni}\ \ \mbox{for all}\ \dot{x}\in
A/P_n.$$ 
Since the length of $\psi_n$ is at most $m$ for each $n$, by Lemma \ref{le220} the inequality (\ref{271}) can not hold for all $n$,
hence $\|\phi-\psi\|\geq1/5$ and $\phi\notin\ue{A}$. 
\end{proof}

\section{The multiplier algebra of a homogeneous C$^*$-algebra}

Recall that a C$^*$-algebra $A$ is called  $n$-subhomogeneous ($n\in\bn$) if $n$
is the maximal dimension of irreducible representations of $A$. Then the intersection
of the kernels of all irreducible representations of dimension at most $n-1$ is an
ideal $J$ of $A$ such that all irreducible representations of $J$ are $n$-dimensional.
$J$ is called the {\em $n$-homogeneous ideal of $A$}; it is the largest ideal of $A$
which is $n$-homogeneous as a C$^*$-algebra.

For an ideal $J$ in $A$ we shall denote by $J^{\perp}$ the annihilator of $J$ in $A$.
Note that the left annihilator is equal to the right annihilator, that is, $aJ=0$ if
and only if $Ja=0$ ($a\in A$).

\begin{lemma}\label{le24} Suppose that $A$ is $n$-subhomogeneous, $J$ is the 
$n$-homogeneous ideal of $A$, $B=A/\ort{J}$, $K$ is the $n$-homogeneous ideal of $B$
and $q:A\to B$ is the quotient map. Then $q(J)=K$ and $K$ is an essential ideal in $B$.
\end{lemma}

\begin{proof}Since $J\cap\ort{J}=0$, $q|J$ is injective, so $q(J)$ is isomorphic to $J$,
hence $n$-homogeneous. Since $q(J)$ is an ideal in $B$, it follows that $q(J)\subseteq K$.
Thus, $J\subseteq q^{-1}(K)$ and then $J+\ort{J}\subseteq q^{-1}(K)$. If $J+\ort{J}\ne q^{-1}(K)$,
then there exists an irreducible representation $\pi$ of $A$ such that $\pi(J+\ort{J})=0$
and $\pi(q^{-1}(K))\ne0$. Since the set $S:=\{[\sigma]\in\hat{A}:\ \dim\sigma\leq n-1\}$
is closed in $\hat{A}$  \cite[4.4.10]{Pe} and $J$
is just the intersection of kernels of representations (the equivalence classes of which
are) in $S$, (the class of) every irreducible representation that annihilates $J$ must be in $S$. Thus 
$[\pi]$ is in $S$, so $\dim\pi<n$. Further, $\pi(\ort{J})=0$ implies that $\pi$ descends
to an irreducible representation $\sigma$ of $B$ (so that $\pi=\sigma q$) and
$\sigma(K)\ne0$, since $\pi(q^{-1}(K))\ne0$. But $\dim\sigma=\dim\pi<n$, which contradicts
the definition of $K$ as the intersection of kernels of all irreducible representations of $B$
of dimension less than $n$.

The ideal $q(J)$ in $B=A/\ort{J}$ is essential, since $aJ\subseteq\ort{J}$ ($a\in A$)
means that in fact $aJ\subseteq J\cap\ort{J}=0$, hence $a\in\ort{J}$.
\end{proof}

If $Z$ is the center of a unital C$^*$-algebra $A$ (or more generally, a C$^*$-subalgebra 
of the center of the multiplier algebra of a not necessarily unital $A$ such that $ZA$ is dense in $A$), $\Delta$ is
the maximal ideal space of $Z$ and for each $t\in \Delta$ we denote by $A(t)$ the quotient
algebra $A(t)=A/(At)$, then for every $x\in A$ the function $t\mapsto\|x(t)\|$ is upper
semicontinuous on $\Delta$ \cite[C.10]{Wi}, \cite{KW} (and vanishes at $\infty$). If these
functions are continuous, then the set $E=\{(t,x(t)):\ t\in\Delta,\ x\in A\}$ can be equipped
with a  topology such that $E$ becomes a C$^*$-bundle with fibers $A(t)$ in the sense of 
\cite[Appendix C]{Wi} or \cite{FD} and $A$ is (isomorphic to) the C$^*$-algebra $\Gamma_0(E)$
of all continuous sections of $E$ vanishing at $\infty$. Since we do not need this topology
here, we only recall that a section of $E$ is a map $s:\Delta\to E$ such that $s(t)\in A(t)$
for all $t\in\Delta$.

The following lemma can be deduced as a special case from a more general result in
\cite{APT}, but we shall sketch a short direct proof. For a C$^*$-bundle $E$ let $\Gamma_b(E)$ be the C$^*$-algebra of all
continuous bounded sections  of $E$ and $\Gamma_0(E)$ the ideal in $\Gamma_b(E)$
consisting of all sections vanishing at $\infty$. 

\begin{lemma}\label{le25}If the fibers of a C$^*$-bundle $E$ over a locally compact space
$\Delta$ are finite dimensional, then $M:=\Gamma_b(E)$ is just the multiplier algebra
of $J:=\Gamma_0(E)$.
\end{lemma}

\begin{proof}For each point $e\in E$ there is a section in $J$ passing
through $e$ and it follows that $J$ is an essential ideal in $M$. It suffices to prove
that for each C$^*$-algebra $A$, which contains $J$ as an essential ideal, the inclusion
$J\to A$ can be extended to a $*$-homomorphism $L:A\to M$. For each $t\in\Delta$ and
$a\in A$ define a map $L_{t,a}$ on the fibre $E_t$ of $E$ by
$$L_{t,a}(s(t))=(as)(t)\ \ (s\in J).$$
Here we have used the fact that each element of $E_t$ is of the form $s(t)$ for some
$s\in J$, but since $s$ is not unique, we need to check that $s(t)=0$ implies $(as)(t)=0$.
This follows from
$$(as)(t)^*(as)(t)=((as)^*(as))(t)\leq\|a\|^2(s^*s)(t)=\|a\|^2s(t)^*s(t),$$
which shows also that $\|L_{t,a}\|\leq\|a\|$. Clearly $L_{t,a}$ is linear and, to check
that $L_{t,a}$ is a left multiplication by an element of $E_t$, it suffices to verify
that $L_{t,a}$ commutes with all right multiplications $R_{z(t)}$ ($z\in J$). For each
$s\in J$ we indeed have
$$L_{t,a}(s(t)z(t))=L_{t,a}((sz)(t))=(asz)(t)=(as)(t)z(t)=L_{t,a}(s(t))z(t).$$

Thus, the function $L(a)$ which sends $t\in\Delta$ to $L_{t,a}$ is a bounded section of 
$E$. To show that it is continuous, choose an approximate unit $(e_k)$ in $J$ and observe
that $L(a)$ is the uniform limit on compact subsets of $\Delta$ of continuous sections
$L(a)e_k=ae_k\in J$. Indeed, for each $t\in\Delta$ and $s\in J$ we have
$$\|(L(a)(t)-(L(a)e_k)(t))s(t)\|=\|(a(1-e_k)s)(t)\|\stackrel{k}{\longrightarrow}0,$$
which implies, since $E_t$ is finite dimensional (with all elements of the form $s(t)$),
that $\|(L(a)(1-e_k))(t)\|\stackrel{k}{\longrightarrow}0$. To show that the convergence is
uniform on compact sets, note that
$$\|(L(a)(1-e_k))(t)\|^2=\|(L(a)(1-e_k)^2L(a)^*)(t)\|\leq\|(L(a)(1-e_k)L(a)^*)(t)\|$$
and that the net of functions $t\mapsto\|(L(a)(1-e_k)L(a)^*)(t)\|$ is decreasing
(since the approximate unit $(e_k)$ is increasing), so Dini's theorem applies.
This shows that $L(a)\in M$ and it can be verified that the map $a\to L(a)$ is a
contractive homomorphism from $A$ to $M$.
\end{proof}

If $J$ is a $n$-homogeneous C$^*$-algebra, then $J$ is (isomorphic to) $\Gamma_0(E)$
for some locally trivial C$^*$-bundle $E$ over $U:=\hat{J}$ by \cite{Fe}, \cite{TT}.
The multiplier algebra $M(J)=\Gamma_b(E)$ is $n$-subhomogeneous by \cite[IV.1.4.6]{Bla},
but in general not $n$-homogeneous as we shall now explain. 

If $E$ is of finite type (that is, if $U$ admits a finite covering by open subsets $U_i$
with $E|U_i$ trivial), then $E$ can be extended to a locally trivial C$^*$-bundle
$F$ over the Stone - Chech compactification $\beta(U)$ \cite[2.9]{Ph} and it follows (since
such a bundle is a direct summand of a trivial  bundle and bounded continuous functions
on $U$ have unique continuous extensions to $\beta(U)$) that
$M(J)=\Gamma_b(E)$ is isomorphic to the C$^*$-algebra $\Gamma(F)$
of all continuous sections of $F$, hence $M(J)$ is $n$-homogeneous in this case.

Conversely, if $M:=M(J)$ is $n$-homogeneous, then by \cite{Fe} $M=\Gamma(F)$ for a locally trivial 
C$^*$-bundle $F$ over the  compact Hausdorff space $\hat{M}\sim\hat{Z}_M$,
where $Z_M$ is the center of $M$, and (by the Dauns - Hofmann theorem) $\hat{Z}_M$
can be identified with $\beta(\hat{J})\cong\beta(\hat{Z}_J)$. Since $J$ is an ideal in $M=\Gamma(F)$, it
follows that $J$ is of the form $J=\{s\in\Gamma(F):\ s|\Lambda=0\}$ for a closed set 
$\Lambda\subseteq\beta(\hat{Z}_J)$ and, considering the characters of the center, $\Lambda$ must
be $\beta(\hat{Z}_J)\setminus\hat{Z}_J$. We conclude that $J=\Gamma_0(F|\hat{Z}_J)$, and
the C$^*$-bundle $F|\hat{Z}_J$ has an extension to a locally trivial C$^*$-bundle $F$ over a
compact space, hence is of finite type by \cite[2.9]{Ph}. Thus we can state the following remark.

\begin{remark}\label{re26} The multiplier algebra of a $n$-homogeneous C$^*$-algebra $J$ ($n\in\bn$)
is $n$-homogeneous if and only if $J$ is of finite type.
\end{remark}

We shall need the fact that for a non-unital $n$-homogeneous C$^*$-algebra $J$ the $n$-homogeneous ideal of 
$M(J)$ is strictly larger than $J$.

\begin{lemma}\label{le27} Let $E$ be a locally trivial C$^*$-bundle with fibers $\matn{\bc}$
($n\in\bn$) over a non-compact, locally compact
space $U$, $J:=\Gamma_0(E)$, $M$ the multiplier C$^*$-algebra of $J$
and $K$ the $n$-homogeneous ideal of $M$. Regard each point $t\in\beta(U)$ (the Stone - Chech
compactification of $U$) as a maximal ideal of the center $Z_M$ of $M$. 
Then $M$ is the C$^*$-algebra of continuous sections of a (not necessarily locally trivial)
C$^*$-bundle $E_0$, with fibers $M(t):=M/(Mt)$, over $\beta(U)$, extending $E$, such that $F:=E_0|\hat{K}$ is locally
trivial. Moreover, at least if $U$ is metrizable, $\hat{K}$ properly contains $U$ (that is, $K$ properly contains $J$).
\end{lemma}

\begin{proof} For each $x\in M$ denote by $x(t)$ the coset of $x$ in $M(t)$. 
The function $\check{x}(t):=\|x(t)\|$ is upper
semicontinuous on $\check{Z}_M=\beta(U)$  \cite[C10]{Wi}. Moreover, $\check{x}$ must be lower
semicontinuous on $U$ as the supremum $\sup\{(xy)^{\check{ }}:\ y\in J,\ \|y\|\leq1\}$
of continuous functions (note that $xy\in J=\Gamma_0(E)$ if $y\in J$). To show that
$\check{x}$ is continuous on all $\beta(U)$,  we may assume that $x\geq0$ (otherwise
just replace $x$ by $|x|$). It suffices now to prove that $\check{x}$ coincides
with the unique continuous extension $\tilde{x}$ of the bounded continuous function
$\check{x}|U$. In other words, we have to show  for each $t^{\prime}\in\beta(U)\setminus U$
and each net $(t_{\nu})\subseteq U$ converging to $t^{\prime}$
the equality
$$\check{x}(t^{\prime})=\lim\check{x}(t_{\nu}).$$
The inequality $\tilde{x}(t^{\prime})\leq\check{x}(t^{\prime})$ follows from the continuity
of $\tilde{x}$ and the upper semicontinuity of $\check{x}$ since the two functions coincide
on the dense set $U$. Suppose that $\tilde{x}(t^{\prime})<\check{x}(t^{\prime})$. Then, by
continuity of $\tilde{x}$,
for a  small $\varepsilon>0$ we have the inequality $\tilde{x}(t)\leq\check{x}
(t^{\prime})-\varepsilon$ for all $t$ in an open neighborhood $V$ of $t^{\prime}$ in 
$\beta(U)$. Choose a continuous
function $f:[0,\infty)\to[0,1]$ such that $f([0,\check{x}(t^{\prime})-\varepsilon])=0$
and $f(\check{x}(t^{\prime}))=1$. Note that for $t\in U\cap V$ the spectrum of $x(t)$ is 
contained in $[0,\check{x}(t^{\prime})-\varepsilon]$, hence $f(x)(t)=f(x(t))=0$ and
therefore by continuity $\widetilde{f(x)}(t)=0$ for all $t\in V$. Further, $(f(x))^{\check{}}
(t^{\prime})=\|f(x)(t^{\prime})\|=\|f(x(t^{\prime}))\|=1$, since $\check{x}(t^{\prime})$
is in the spectrum of $x(t)$ and $f(\check{x}(t^{\prime}))=1$. Thus, replacing $x$ by $f(x)$, we achieve that $\tilde{x}(t)
=0$ if $t\in V$ and $\check{x}(t^{\prime})=1$. Choosing a continuous function $\chi$ on $\beta(U)$
with values in $[0,1]$, supported in $V$ and with $\chi(t^{\prime})=1$, and replacing $x$
by $\chi x$ (where $\chi$ is regarded as an element of $Z_M$ by the Dauns - Hofmann theorem),
we find an element $x\in M$ such that $\tilde{x}(t)=0$ for all $t\in\beta(U)$ (hence $x=0$)
and $\check{x}(t^{\prime})=1$, which is a contradiction. The just proved continuity of $\check{x}$ means
that $M$ is the C$^*$-algebra of continuous sections a C$^*$-bundle $E_0$ over $\beta(U)$ with 
fibers $M(t)$ \cite[Appendix C]{Wi}. 

In general the map $\zeta:\hat{M}\to\check{Z}_M=\beta(U)$, $\zeta([\pi])=\ker(\pi|Z_M)$,
is continuous, but since the functions $\check{x}$
($x\in M$) are continuous, this map is also open \cite[C.10]{Wi}. Since $J$ and $K$ ($J\subseteq K$)
are essential ideals in $M$, one can verify the inclusion of the centers $Z_J\subseteq Z_K
\subseteq Z_M$ as  ideals in $Z_M$. Further, $\zeta(\hat{K})=\check{Z}_K$.
(More precisely, denoting for each $[\pi]\in\hat{K}$ by $\tilde{\pi}$ the unique extension of
$\pi$ to the irreducible representation of $M$, $\tilde{\pi}|Z_K=\pi|Z_K$.) Since $K$ is
$n$-homogeneous, we may identify $\hat{K}$ with $\check{Z}_K$, that is, $\zeta$ maps $\hat{K}$
onto $\check{Z}_K\subseteq\beta(U)$ homeomorphically, and we may regard $\hat{K}$ as an 
open subset in $\beta(U)$. Since $K$ is $n$-homogeneous, for
each $t\in\check{Z}_{K}$ there is (up to a unitary equivalence) a unique irreducible representation 
$\pi_t$ of $K$ such that $\ker(\pi|Z_K)=t\cap Z_K$. Then the extension $\tilde{\pi}_t$ of
$\pi_t$ to $M$ is the unique irreducible representation $\sigma$ of $M$ with $\ker(\sigma|Z_M)
=t$. (Namely, $\ker(\sigma|Z_M)=t$ implies that $\ker(\sigma|Z_K)=t\cap Z_K$, hence 
$\sigma|K$ must coincide, up to a unitary equivalence, with $\pi_t$, since irreducible
representations of a homogeneous C$^*$-algebra $K$ are determined by their restrictions to the
center. This implies $\sigma=\tilde{\pi}_t$, since extension of nondegenerate representations from ideals
are unique.) Since each $Mt$ is an intersection of primitive ideals, it follows that 
$Mt$ must be a primitive ideal in $M$ (for  by the above there is only one primitive ideal 
containing $t$) and $M/(Mt)\cong\matn{\bc}$ for all $t\in\check{Z}_{K}$. Further, 
if $t\in\beta(U)\setminus \check{Z}_K$, then $Mt$ must be the intersection of kernels of
certain irreducible representations $\pi$ with $[\pi]\in\hat{M}\setminus\hat{K}$ only. 
It follows that for a section $x\in M$ we have that $x(t)=0$ for all 
$t\in\beta(U)\setminus\check{Z}_{K}$ if
and only if $\pi(x)=0$ for all  $[\pi]\in\hat{M}\setminus\hat{K}$. This means that the ideal $\Gamma_0(E_0|\hat{K})$ in 
$\Gamma(E_0)=M$ must be $K$. Since $K$ is $n$-homogeneous, it follows (using
\cite{Fe}) that $F:=E_0|\hat{K}$ must be locally trivial.
Finally, since $K$ contains $J$ as an ideal, 
$J=\{s\in K:\ s|(\hat{K}\setminus U)=0\}=\Gamma_0(F|U)$ \cite[II.14.8]{FD}, hence $F|U\cong E$.

To show that $\hat{K}$ properly contains $U$, choose a
sequence $(t_k)$ in $U$ with no limit points in $U$ (recall that $U$ is assumed metrizable)
and sections 
$s_{ij}\in M=\Gamma_b(E)$
such that $s_{ij}(t_k)$ ($i,j=1,\ldots,n$) are the matrix units in the fibers 
$E_{t_k}\cong
\matn{\bc}$. For each section $s\in M$ we  expand
$s(t_k)=\sum_{i,j=1}^n\alpha_{ij}(t_k)s_{ij}(t_k)\ \ (\alpha_{ij}(t_k)\in\bc),$
extend each (bounded) sequence $(\alpha_{ij}(t_k))_k$ to a continuous function $\alpha_{ij}$ on
$\beta(U)$, choose a  limit point
$t_0\in\beta(U)\setminus U$ of $(t_k)$  and set
$$\pi_{t_0}(s):=\sum_{i,j=1}^n\alpha_{i,j}(t_0)e_{ij}=[\alpha_{ij}(t_0)]\in\matn{\bc},$$
where $e_{ij}$ are the standard matrix units in $\matn{\bc}$. This defines a representation
$\pi_{t_0}$ of $M$ into $\matn{\bc}$ ($\pi_{t_0}(s)$ is a kind of a limit point of $(s(t_k))$), which is surjective (hence irreducible) since
$\pi_{t_0}(s_{ij})=e_{ij}$.
If $[\pi_{t_0}]$ were not in $\hat{K}$, then $\pi_{t_0}(K)=0$, which would imply (by the
definition of $K$) that 
$\ker\pi_{t_0}$ is in the closure of the set of kernels of all irreducible representations
of $M$ of dimension less than $n$, which is impossible since this set is closed.
\end{proof}

\section{A reduction to locally homogeneous C$^*$-algebras}

\begin{lemma}\label{le28} If a  separable $n$-subhomogeneous C$^*$-algebra $A$ is not a direct sum of 
homogeneous C$^*$-algebras, then $\icb{A}\not\subseteq\ue{A}$.
\end{lemma}

Since the proof of the Lemma occupies the entire section, it will be divided into
several steps. Let $J$ be the $n$-homogeneous ideal of $A$, $U$ the primitive spectrum of
the center $Z_J$ of $J$, $E$ the locally trivial C$^*$-bundle over $U$ such that 
$J=\Gamma_0(E)$, and $M=M(J)=\Gamma_b(E)$ the multiplier C$^*$-algebra of $J$. 
If $J$ is
unital, then $A$ is isomorphic to $J\oplus(A/J)$, where $A/J$ is $m$-subhomogeneous for 
some $m<n$, and the proof reduces to a smaller degree of subhomogeneity. So by an induction we may
assume that $J$ is not unital, hence $U$ is not compact. We shall show that in this case
$\icb{A}\not\subseteq\ue{A}$. By Lemma
\ref{le27} the $n$-homogeneous ideal $K$ of $M$ properly contains $J$ and the corresponding
locally trivial C$^*$-bundle $F$ over the open subset $\hat{K}$ of $\beta(U)$ 
(so that $K=\Gamma_0(F)$)
extends $E$, while $M=\Gamma(E_0)$ for a (not necessarily locally trivial) C$^*$-bundle 
$E_0$ over $\beta(U)$ extending $F$. We  denote by $Z_K$ and $Z_M$ the centers of $K$ and $M$, identify $\hat{K}$ 
and $\hat{J}$ with $\check{Z}_K$ and $\check{Z}_J$
(respectively) and regard them as open subsets of $\check{Z}_M=\beta(U)$. Choose  $t_0\in\hat{K}\setminus\hat{J}$ and an open neighborhood $V$ of
$t_0$ in $\beta(U)$ such that $\overline{V}\subseteq\hat{K}$ and $F|\overline{V}$ is trivial.
Using a fixed isomorphism $E_0|\overline{V}=F|\overline{V}\cong\overline{V}\times\matn{\bc}$, we shall identify
the two bundles over $V$.

\medskip
{\bf Suppose first that the ideal $J$ in $A$ is essential.} Then we may regard $A$ as a 
C$^*$-subalgebra
of $M$. Since all $n$-dimensional irreducible representations of $A$ are (up to a unitary
equivalence) evaluations at points of $U$, for each $t\in\overline{V}\setminus U$ the 
evaluation $\pi_t$  of sections of $E_0$ at $t$ must be reducible as a representation of $A$. 
Let $m$ be the maximal dimension of
irreducible subrepresentations of $\pi_t|A$ as $t$ ranges over $\overline{V}\setminus U$
and let $t_1\in\overline{V}\setminus U$ be a point where this maximum is attained.
Then (up to a unitary equivalence) $\pi_{t_1}|A$ has the form
\begin{equation}\label{30}\pi_{t_1}(a)=\left[\begin{array}{cc}
\sigma_{t_1}^{(k)}(a)&0\\
0&\rho_{t_1}(a)\end{array}\right]\ \ (a\in A),\end{equation}
where $\sigma_{t_1}:A\to {\rm M}_m(\bc)$ is an irreducible representation, $k\in\bn$ and 
$\rho_{t_1}:A\to{\rm M}_{n-km}$ is a representation disjoint from $\sigma_{t_1}$. Denote
by $e_{ij}$ ($i,j=1,\ldots,m$) the standard matrix units in  ${\rm M}_m(\bc)$. By 
\cite[4.2.5]{Dix} there exist $a_{ij}\in A$ such that $\pi_{t_1}(a_{ij})=e_{ij}^{(k)}
\oplus0$ (relative to the decomposition (\ref{30})). By continuity, if $t$ is close to $t_1$,
$\pi_t(a_{ij})$ will be approximately matrix units in ${\rm M}_m(\bc)$ and well known arguments
(using functional calculus and polar decomposition, similarly as in \cite[Section 12.1]{KR}) 
show that there exist $b_{ij}\in A$
such that $\pi_t(b_{ij})$ ($i,j=1,\ldots,m$) are $m\times m$ matrix units in $\matn{\bc}$;
in other words, $\pi_t(A)$ contains a copy of ${\rm M}_m(\bc)$ for all $t$ in a neighborhood
$W\subseteq\overline{W}\subseteq V$ of $t_1$. It follows now by maximality of $m$ that
(up to a conjugation with a unitary $u\in C(W,\matn{\bc})$) $\pi_t|A$ has the form
\begin{equation}\label{31}a(t):=\pi_t(a)=\left[\begin{array}{cc}
\sigma_t(a)&0\\
0&\theta_t(a)\end{array}\right]\ \ (a\in A,\ t\in W\setminus U),
\end{equation}
where $\sigma_t:A\to {\rm M}_m(\bc)$ is an irreducible and $\theta:A\to {\rm M}_{n-m}(\bc)$
a (possibly degenerate) representation.

Choose a continuous function $\chi$ on $\beta(U)\setminus U$, supported in $W\setminus U$,
with values in $[0,1]$ and $\chi(t_1)=1$. Let $v\in{\rm M}_{m,n-m}(\bc)$ be any matrix
with $\|v\|=1$. Since $M=\Gamma(E_0)$ and 
$J=\Gamma_0(E)=\{s\in\Gamma(E_0):\ s|(\beta(U)\setminus U)=0\}$,  
$M/J=\Gamma(E_0|(\beta(U)\setminus U))$ (using the Tietze extension
theorem for sections of Banach bundles \cite[II.14.8]{FD}). Define a section $s\in M/J$ 
on $\beta(U)\setminus U$ by
\begin{equation}\label{32}s(t)=\left[\begin{array}{cc}
0&\chi(t)v\\
0&0\end{array}\right]\  \mbox{if}\ t\in W\setminus U\ \mbox{and}\ s(t)=0\ \mbox{if}\ 
t\in(\beta(U)\setminus U)\setminus W
\end{equation}
and let $b\in M$ be any lift of $s$ (that is, a continuous extension of $s$ to a section
of $E_0$). Finally, let $\phi:A\to M$ be the 
twosided multiplication $x\mapsto bxb$.

\medskip
{\em Proof that $\phi(A)\subseteq A$ and that $\phi$ preserves ideals.} Given $a\in A$, the value $\phi(a)(t)$ of $\phi(a)\in M$ at each $t\in\beta(U)\setminus U$
is $0$. Indeed, $b(t)=s(t)=0$ if $t\in(\beta(U)\setminus U)\setminus W$, while for $t\in W\setminus U$ we have that
$\phi(a)(t)=b(t)a(t)b(t)=s(t)\pi_t(a)s(t)=0$, as can be verified by performing the matrix
multiplication with $\pi_t(a)$ and $s(t)$ of the form (\ref{31}) and (\ref{32}). This implies
that $\phi(a)\in J$; in particular $\phi$ maps $A$ into $A$. To show that $\phi$ preserves
all ideals in $\id{A}$, let $(e_k)$ be an approximate unit in the $n$-homogeneous ideal $J$. 
Note that (since $\phi(a)
\in J$)
$$\phi(a)=\lim e_k\phi(a)e_k=\lim(e_kb)a(be_k),$$
where the two sided multiplications $a\mapsto (e_kb)a(be_k)$ preserve the ideals since
$e_kb$ and $be_k$ are in $J\subseteq A$. Thus $\phi$ is a pointwise limit of maps preserving 
ideals,
so $\phi$ must preserve (closed) ideals.  

\medskip
{\em Proof that $\phi\notin\ue{A}$.} First  we shall `localize' the proof to $W$ (to work with
matrix valued functions instead of bundles), then
we shall show by an explicit computation that $\phi\notin\ue{A}$. 

Let $J_W=\{a\in M:\ a(t)=0\ \forall t\in
\overline{W}\}$ and let $\phi_W$ be the map on $A_W:=A/(J_W\cap A)$ induced by $\phi$. 
Note that $A_W$ is (naturally isomorphic to) a C$^*$-subalgebra of $M/J_W=
\Gamma(E_0|\overline{W})=\Gamma_0(F|\overline{W})=C(\overline{W},\matn{\bc})$, and $\phi_W$
is just the twosided multiplication 
$$\phi_W(x)=dxd\ \ (x\in A_W\subseteq C(\overline{W},\matn{\bc})),$$
where $d$ is the coset of $b$ in $M/J_W$. As an element of $C(\overline{W},\matn{\bc})$, decomposing
$\matn{\bc}$ into blocks according to  (\ref{31}), $d$ can be represented by a block matrix of
continuous functions
$$d=\left[\begin{array}{cc}
d_{11}&d_{12}\\
d_{21}&d_{22}\end{array}\right],$$
where (by the definitions of $b$ and $s$) $d_{11}(t_1)=0$, $d_{21}(t_1)=0$, $d_{22}(t_1)=0$
and $d_{12}(t_1)=v$. It follows now from $\phi_W(x)=dxd$ that
$$\phi_W(x)(t_1)=\left[\begin{array}{cc}
0&vx_{21}(t_1)v\\
0&0\end{array}\right]\ \mbox{for all}\ x=\left[\begin{array}{cc}
x_{11}&x_{12}\\
x_{21}&x_{22}\end{array}\right]\ \mbox{in}\ A_W\subseteq C(\overline{W},\matn{\bc}).$$
Given $\varepsilon>0$, by continuity of  functions $d_{ij}$  
(that is, since $\|d(t)-d(t_1)\|$ is small if $t\in W$ is close to $t_1$) there exists a neighborhood
$W_1\subseteq W$ of $t_1$ such that we have uniformly for all $x=[x_{ij}]\in A_W$ 
with $\|x\|\leq1$ the estimate
\begin{equation}\label{33}\|\phi_W(x)(t)-\left[\begin{array}{cc}
0&vx_{21}(t)v\\
0&0\end{array}\right]\|<\varepsilon\ \ \mbox{for all}\ t\in W_1.\end{equation}

The evaluation $\pi_{t_1}$  maps  $A$ into block diagonal matrices 
according to (\ref{31}), but we shall need  the same for 
$M(A)$ ($\pi_{t_1}$ can be degenerate). Since
$J$ is essential in $A$, hence also in $M(A)$, we have that $M(A)\subseteq M(J)=M$, hence
each $f\in M(A)$ can be represented over $W$ by a $2\times 2$ block matrix $f|W=[f_{ij}]$
in accordance with the decomposition (\ref{31}). Let 
$p\in\matn{\bc}$ be the
projection onto $[\sigma_{t_1}^{(k)}(A)\bc^n]$ (where $\sigma_{t_1}$ is as in (\ref{30})).
Then $p\in\pi_{t_1}(A)$ since $\rho_{t_1}$ and $\sigma_{t_1}^{(k)}$ are disjoint. 
With respect to the  decomposition  (\ref{31}), $p$ has 
the form $p=1\oplus q$, 
where $1$ is the $m\times m$ identity matrix and $q$ is a projection. Since 
$f(t_1)p\in\pi_{t_1}(A)$ and $pf(t_1)\in\pi_{t_1}(A)$ and matrices in $\pi_{t_1}(A)$
are block - diagonal, a matrix multiplication shows 
that $f_{21}(t_1)=0$ and $f_{12}(t_1)=0$. Thus $\pi_{t_1}(M(A))$ consists of block - diagonal
matrices only.

\medskip
Suppose that there exists $\psi\in\e{A}$ with $\|\psi-\phi\|<\varepsilon$,
hence 
\begin{equation}\label{34}\|\psi_W-\phi_W\|<\varepsilon,\end{equation}
where $\psi_W$ is the map induced on $A_W$ by $\psi$.
Then $\psi$ is of the form
$$\psi(x)=\sum_{k=1}^{\ell}a^kxb^k\ \ (x\in A),$$
where $a^k,b^k\in M(A)\subseteq M$.  By the previous paragraph 
$a^k(t_1)=a^k_{11}(t_1)\oplus a^k_{22}(t_1)$ and 
$b^k(t_1)=b^k_{11}(t_1)\oplus b^k_{22}(t_1)$ are block diagonal.
Now, for matrices of the form
\begin{equation}\label{35}x=\left[\begin{array}{cc}
0&0\\
x_{21}&0\end{array}\right]\end{equation}
we have that $\sum_{k=1}^{\ell}a^k(t_1)xb^k(t_1)$ is of the form
$$\left[\begin{array}{cc}
0&0\\
\sum_{k=1}^{\ell}a^k_{22}(t_1)x_{21}b^k_{11}(t_1)&0
\end{array}\right],$$
hence by continuity of the coefficients $a^k$ and $b^k$ (on $W$) there exists a neighborhood $W_2\subseteq W$ of $t_1$ such that
\begin{equation}\label{36}\|\psi_W(x)(t)-\left[\begin{array}{cc}
0&0\\
\sum_{k=1}^{\ell}a^k_{22}(t)x_{21}(t)b^k_{11}(t)&0\end{array}\right]\|<\varepsilon\ \mbox{
for all}\ t\in W_2
\end{equation} 
uniformly for all $x\in A_W$ of the form (\ref{35}) with $\|x\|\leq1$. From (\ref{33}), (\ref{34})
and (\ref{36}) we conclude that 
\begin{equation}\label{37}\|\left[\begin{array}{cc}
0&vx_{21}(t)v\\
-\sum_{k=1}^{\ell}a^k_{22}(t)x_{21}(t)b^k_{11}(t)&0\end{array}\right]\|<3\varepsilon\end{equation}
for all $t\in W_1\cap W_2$ and $x\in A_W$ of the form (\ref{35}) with $\|x\|\leq1$.
But, for each  $t\in W_1\cap W_2\cap U$, we have that $A_W(t)=\pi_t(A)=\matn{\bc}$ (since
already $J(t)=\matn{\bc}$),
hence we may choose $x\in A_W$ of the form (\ref{35}) so that $\|x_{21}(t)\|=1$ and
$\|vx_{21}(t)v\|=1$ (for a fixed $t$), which contradicts (\ref{37}) if $\varepsilon<1/3$. 
Thus  $\phi\notin\ue{A}$. This proves the lemma
in the case $J$ is essential in $A$.

\medskip
{\bf A reduction to the case when $J$ is essential.} Let $B=A/\ort{J}$, $q:A\to B$ the quotient map and $K=q(J)$. By Lemma \ref{le24} $K$ is
the $n$-homogeneous ideal of $B$ and is an essential ideal in $B$. By what we have already proved above, there 
exists $b\in M(K)$ such that the twosided multiplication $\phi(x)=bxb$ maps $B$ into $K$
and $\phi\in\icb{B}\setminus\ue{B}$. Define $\phi_0:A\to A$ as the composition
$$\phi_0=(q|J)^{-1}\phi q.$$
Then $\phi_0(A)\subseteq J$. To show that $\phi_0$ preserves ideals of $A$, let $(e_k)$ be
an approximate unit in $J$ and choose $a\in M(J)$ so that $\tilde{q}(a)=b$, where $\tilde{q}$
is the extension to $M(J)\to M(K)$ of the isomorphism $q|J:J\to K$. Then for $x\in A$
$$\begin{array}{lll}
\phi_0(x)&=&\lim e_k\phi_0(x)e_k=\lim e_k(q|J)^{-1}(bq(x)b)e_k\\
&=&\lim (q|J)^{-1}(q(e_k)bq(x)bq(e_k))
 =\lim(q|J)^{-1}q(e_kaxae_k)\\
&=&\lim(e_ka)x(ae_k),\end{array}$$
hence $\phi_0(x)$ is in (the closed twosided) ideal generated by $x$ since $e_ka\in J\subseteq A$. To show that
$\phi_0\notin\ue{A}$, assume the contrary, that for each $\varepsilon>0$ there exists
$\psi\in\e{A}$ with $\|\phi_0-\psi\|\leq\varepsilon$. Denote by $\dot{\psi}$ the elementary
operator on $B$ induced by $\psi$, so that $\dot{\psi}q=q\psi$. Then for each $x$ in the unit
ball of $A$ we have that $\|\phi_0(x)-\psi(x)\|\leq\varepsilon$, which implies that
$\|\phi q(x)-\dot{\psi}q(x)\|=\|q(\phi_0(x)-\psi(x))\|\leq\varepsilon$. Since $q$ maps
the closed unit ball of $A$ onto that of $B$, it follows that $\|\phi-\dot{\psi}\|\leq
\varepsilon$. But this would imply that $\phi\in
\ue{B}$, a contradiction. \rightline{$\Box$}

\medskip

Combining Lemmas \ref{le23} and \ref{le28} with what we have proved in the Introduction proves
Theorem \ref{th11}.

The author does not know if Theorem \ref{th11} holds also for nonseparable C$^*$-algebras.

\end{document}